\documentclass[12pt]{article}
\usepackage{amsmath}
\usepackage{latexsym}
\usepackage{amssymb}
\newtheorem{thm}{Theorem}[section]
\newtheorem{la}[thm]{Lemma}
\newtheorem{Defn}[thm]{Definition}
\newtheorem{Remark}[thm]{Remark}
\newtheorem{prop}[thm]{Proposition}
\newtheorem{cor}[thm]{Corollary}
\newtheorem{Example}[thm]{Example}
\newtheorem{Number}[thm]{\!\!}
\newenvironment{defn}{\begin{Defn}\rm}{\end{Defn}}
\newenvironment{example}{\begin{Example}\rm}{\end{Example}}
\newenvironment{rem}{\begin{Remark}\rm}{\end{Remark}}
\newenvironment{numba}{\begin{Number}\rm}{\end{Number}}
\newenvironment{proof}{{\noindent\bf Proof.}}%
                  {\nopagebreak\hspace*{\fill}$\Box$\medskip\medskip\par}   
\newcommand{\Punkt}{\nopagebreak\hspace*{\fill}$\Box$}
\newcommand{\wb}{\overline}
\newcommand{\ve}{\varepsilon}
\newcommand{\at}{\symbol{'100}}
\newcommand{\tensor}{\otimes}
\newcommand{\mto}{\mapsto}
\newcommand{\isom}{\cong}
\newcommand{\N}{{\mathbb N}}
\newcommand{\R}{{\mathbb R}}

\newcommand{\cO}{{\cal O}}
\newcommand{\cV}{{\cal V}}
\newcommand{\dl}{{\displaystyle \lim_{\longrightarrow}}}
\newcommand{\sub}{\subseteq}
\newcommand{\cB}{{\cal B}}
\newcommand{\cF}{{\cal F}}
\newcommand{\cK}{{\cal K}}
\newcommand{\cR}{{\cal R}}
\newcommand{\cL}{{\cal L}}
\newcommand{\cT}{{\cal T}}
\DeclareMathOperator{\conv}{conv}
\DeclareMathOperator{\Supp}{supp}

\DeclareMathOperator{\lcx}{lcx}
\begin{document}
$\;$\\[-24mm]
\begin{center}
{\Large\bf Upper bounds for continuous seminorms and\\[2mm]
special properties of bilinear maps}\\[7mm]
{\bf Helge Gl\"{o}ckner}\vspace{4mm}
\end{center}
\begin{abstract}\noindent
If $E$ is a locally convex topological vector space,
let $(P(E),\preceq)$ be the pre-ordered set of all continuous seminorms
on~$E$.
We study, on the one hand,
for $\theta$ an infinite cardinal
those locally convex spaces~$E$
which have the $\theta$-neighbourhood property introduced by E. Jord\'{a},
meaning that all sets $M$ of continuous
seminorms of cardinality $|M|\leq \theta$
have an upper bound in $P(E)$.
On the other hand, we study
bilinear maps $\beta\colon E_1\times E_2\to F$
between locally
convex spaces which admit ``product estimates''
in the sense that for all $p_{i,j}\in P(F)$,
$i,j=1,2,\ldots$, there exist $p_i\in P(E_1)$
and $q_j\in P(E_2)$ such that $p_{i,j}(\beta(x,y))\leq p_i(x)q_j(y)$
for all $(x,y)\in E_1\times E_2$.
The relations\
between these concepts
are explored, and examples given.
The main applications concern spaces
$C^r_c(M,E)$ of vector-valued test functions on manifolds.\vspace{3mm}
\end{abstract}
{\footnotesize {\em Classification}:
46A03,
46F05
%
 (primary);
%
22D15,
%
22E30,
%
42A85,
%
44A35,
%
46A11,
%
46A13,
%
46A32,
%
46E25,
%
46H05,
%
46M05.\\[1mm]
%
%
}
\section{Introduction}
Primarily, this article is devoted to a strengthened
continuity property for bilinear maps which arose recently
in the study of convolution of vector-valued test functions.
In addition, it describes relations between this notion and the countable neighbourhood property,
and discusses further applications of the latter (and the $\theta$-neighbourhood property).\\[3mm]
{\bf Neighbourhood properties.}
For~$E$ a locally convex space,
we obtain a pre-order $\preceq$ on the set $P(E)$ of all continuous seminorms on~$E$
by declaring $p\preceq q$ if $p\leq Cq$ pointwise for some $C>0$.
The space $E$ is said to have the \emph{countable neighbourhood property} (or \emph{cnp}, for short)
if each countable set of
continuous
seminorms has an upper bound in $(P(E),\preceq)$
(see \cite{Fl2} and the references therein).
Likewise,
given an infinite cardinal number~$\theta$,
the space $E$ is said to have the \emph{$\theta$-neighbourhood property}
(of \emph{$\theta$-np}, for short)
if for each set $M$ of continuous seminorms on~$E$ of cardinality
$|M|\leq\theta$, there exists a continuous seminorm~$q$ on~$E$
such that $p\preceq q$ for all $p\in M$ (see \cite[Definition~4.4]{Jor}).\\[2mm]
Besides classical studies (see \cite{Bon}, \cite{Fl2} and the references therein),
the countable neighbourhood property also occurred more recently in
the study of the tensor algebra $\cT(E)$  of a locally convex space~$E$.
Topologize $\cT(E):=\bigoplus_{n\in \N_0}\cT^n(E)$
as the locally convex direct
sum of the projective tensor powers $\cT^0(E):=\R$, $\cT^1(E):=E$,
$\cT^{n+1}(E):=\cT^n(E)\tensor_\pi E$ of $E$.
Answering a question by K.-H. Neeb \cite[Problem~VIII.5]{Nee},
it was shown
that $\cT(E)$ is a topological algebra
(i.e., the bilinear tensor multiplication $\cT(E)\times \cT(E)\to \cT(E)$
is jointly
continuous)
if and only if~$E$ has the cnp \cite[Theorem~B]{Glo}.\\[3mm]
{\bf Product estimates.}
Following \cite{BaG},
a bilinear map $\beta\colon E_1\times E_2\to F$
between locally convex spaces is said to \emph{admit product
estimates} if, for each double sequence $(p_{i,j})_{i,j\in \N}$
of continuous seminorns $p_{i,j}$ on~$F$,
there exists a sequence $(p_i)_{i\in \N}$ of continuous seminorms on~$E_1$
and a sequence $(q_j)_{j\in \N}$
of continuous seminorms on~$E_2$ such that
\[
(\forall i,j\in \N,\; x\in E_1, \; y\in E_2)\quad p_{i,j}(\beta(x,y))\leq p_i(x)q_j(y)\,.
\]
If $\beta$ admits product estimates, then $\beta$ is continuous.\footnote{If $p$ is a continuous seminorm on~$F$,
set $p_{i,j}:=p$ for all $i,j\in \N$, and find corresponding $p_i,q_j$.
Then $p(\beta(x,y))\leq p_1(x)q_1(y)$ for all $x\in E_1$, $y\in E_2$.}
However, a continuous bilinear map need not admit product estimates
(see Section~\ref{secno}).\footnote{If $\beta$ is continuous, given $p_{i,j}$ as before
we can still find continuous seminorms~$P_{i,j}$ on~$E_1$
and~$Q_{i,j}$ on~$E_2$ such that $p_{i,j}(\beta(x,y))\leq P_{i,j}(x)Q_{i,j}(y)$.
However, in general one cannot choose $P_{i,j}$ independently of~$j$,
nor $Q_{i,j}$ independently of $i$.}
Thus, the existence of product estimates
can be regarded as a \emph{strengthened continuity property} for
bilinear maps.\\[3mm]
The concept of product estimates first arose in the study of convolution of
vector-valued test functions. Consider the following setting
(to which we shall return later):
\begin{numba}\label{setting}
Let $b\colon E_1\times E_2\to F$ be a continuous bilinear map
between locally convex spaces such that $b\not=0$.
Let $r,s,t\in \N_0\cup\{\infty\}$
with $t\leq r+s$.
If $r=s=t=0$, let $G$ be a locally compact group;
otherwise, let $G$ be a Lie group.
Let $\lambda_G$ be a left Haar measure on $G$.
If $G$ is discrete, we need not impose any completeness assumptions on $F$.
If $G$ is metrizable and not discrete, we assume that $F$ is sequentially complete.
If $G$ is not metrizable (and hence not discrete either),
we assume that $F$ is complete.
These conditions
ensure the existence of the integrals needed to define the convolution
$\gamma*_b\eta\colon G\to F$
of $\gamma\in C^r_c(G,E_1)$ and $\eta\in C^s_c(G,E_2)$
via
\begin{equation}\label{defco}
(\gamma*_b\eta)(x):=\int_G b(\gamma(y), \eta(y^{-1}x))\, d\lambda_G(y)\qquad \mbox{for $\, x\in G$.}
\end{equation}
Then $\gamma*_b\eta\in C^{r+s}_c(G,F)$ (see \cite[Proposition~2.2]{BaG}),
whence
\begin{equation}\label{defbeta}
\beta_b\colon
C^r_c(G, E_1)\times C^s_c(G, E_2)\to C^t_c(G, F)\,,\quad
(\gamma,\eta)\mapsto \gamma*_b\eta
\end{equation}
makes sense.
\end{numba}
For $G$ compact, $\beta_b$ is always continuous \cite[Corollary~2.3]{BaG}.
If $G$ is an infinite discrete group, then $\beta_b$ is continuous
if and only if~$G$ is countable and~$b$ admits product estimates \cite[Proposition~6.1]{BaG}.
The main result of \cite{BaG} reads:\\[4mm]
{\bf Theorem A.} \emph{If $G$ is neither discrete nor compact,
then the convolution map~$\beta_b$ from} (\ref{defbeta}) \emph{is
continuous if and only if}
(a), (b) \emph{and} (c) \emph{are satisfied}:
\begin{itemize}
\item[(a)]
\emph{$G$ is $\sigma$-compact};
\item[(b)]
\emph{If $t=\infty$, then also $r=s=\infty$}; \emph{and}
\item[(c)]
\emph{$b$ admits product estimates.}\vspace{2mm}
\end{itemize}
{\bf Structure of the article and main results.}
After some preliminaries
(Section~\ref{secprel}),
we recall various examples of spaces with neighbourhood properties,
and some permanence properties of the class of spaces
possessing the $\theta$-np (Section~\ref{secinherit}).
In Section~\ref{pivot},
we prove two simple, but essential results,
which link the concepts discussed in this article:
If $E_1$, $E_2$ and $F$ are locally convex spaces
and $F$ or both of $E_1$ and $E_2$ have the cnp,
then every continuous bilinear map $E_1\times E_2\to F$
admits product estimates (see Propositions~\ref{cubcuseful}
and \ref{newprop}).
This immediately gives a large supply of mappings admitting product estimates.
In Section~\ref{secno}, we describe two simple concrete
examples of continuous bilinear maps
for which it can be shown by hand that they do not admit product estimates.
Section~\ref{sectest} provides basic background
concerning the topology on spaces of vector-valued test
functions, for later use.
Sections~\ref{secscavect} and \ref{secconvol}
are devoted to the proofs
of more difficult theorems.
If $M$ is a Hausdorff topological space, let $\theta(M)$ be the smallest
cardinal of a cover of $M$ by compact sets
(the compact covering number of~$M$).
We show:\\[4mm]
{\bf Theorem B.}
\emph{Let $E$ be a locally convex space
and $r\in \N_0\cup\{\infty\}$.
If $r=0$, let $M$ be a paracompact, locally compact, non-compact topological
space; if $r>0$, let $M$ be a metrizable, non-compact,
finite-dimensional $C^r$-manifold.
Then
\[
\Psi_{c,E}\colon C^r_c(M)\times E \to C^r_c(M,E),\quad (\gamma,v)\mto \gamma v
\]
is a hypocontinuous bilinear map.
The map $\Psi_{c,E}$ is continuous if and only if~$E$ has the $\theta(M)$-neighbourhood property.
If $E$ is metrizable,
then $\Psi_{c,E}$ is continuous if and only if~$E$ is normable.}\\[4mm]
Finally, we obtain
a characterization of those $(G,r,s,t,b)$ for which the\linebreak
convolution map
$\beta_b$ admits product
estimates.\\[4mm]
{\bf Theorem~C.}
\emph{Let $(G,r,s,t,b)$
and $\beta_b\colon C^r_c(G,E_1)\times C^s_c(G,E_2)\to C^t_c(G,F)$ be as in} \ref{setting}.
\emph{Then $\beta_b$ has the following properties}:
\begin{itemize}
\item
\emph{If $G$ is finite, then $\beta_b$ is always continuous.
Moreover, $\beta_b$ admits product estimates if and only if $b$ does.}
\item
\emph{If $G$ is an infinite discrete group, then $\beta_b$ admits product estimates
if and only if $\beta_b$ is continuous, which holds if and only if $G$ is countable and $b$
admits product estimates.}
\item
\emph{If $G$ is an infinite compact group, then $\beta_b$ is always continuous.
Moreover, $\beta_b$ admits product estimates
if and only if the conditions} (a), (b) \emph{and}~(c)
\emph{from Theorem}~A \emph{are satisfied.}
\item
\emph{If $G$ is neither compact nor discrete,
then $\beta_b$ admits product estimates if and only if $\beta_b$ is continuous,
which holds if and only if} (a), (b) \emph{and}~(c)
\emph{from Theorem}~A \emph{are satisfied.}\vspace{1mm}
\end{itemize}
For example, consider
a compact, non-discrete Lie group~$G$.
Then the\linebreak
convolution map
$C^0(G)\times C^\infty(G)\to C^\infty(G)$
is continuous but does not admit product estimates.\\[4mm]
In the final section, we show by example
that product estimates can also be available for non-degenerate
bilinear maps on locally convex spaces which do not admit
continuous norms.
\section{Preliminaries and basic facts}\label{secprel}
\emph{Generalities.}
We write $\N=\{1,2,\ldots\}$ and $\N_0:=\N\cup\{0\}$.
By a \emph{locally convex space}, we mean a Hausdorff locally convex real
topological vector space.
A map between topological spaces
is called a \emph{topological embedding} if it is a homeomorphism
onto its image. If $E$ is a vector space
and $p$ a seminorm on~$E$,
define
$B^p_r(x):=\{y\in E\colon p(y-x)<r\}$
and
$\wb{B}^p_r(x):=\{y\in E\colon p(y-x)\leq r\}$
for $r>0$ and $x\in E$.
Let $E_p:=E/p^{-1}(0)$ be the associated
normed space, with the norm $\|.\|_p$
given by $\|x+p^{-1}(0)\|_p:=p(x)$.
If $X$ is a set and $\gamma\colon X\to E$
a map,
we define $\|\gamma\|_{p,\infty}:=\sup_{x\in X}p(\gamma(x))$.
If $(E,\|.\|)$ is a normed space and $p=\|.\|$,
we write $\|\gamma\|_\infty$ instead of $\|\gamma\|_{p,\infty}$.\\[2.7mm]
\emph{Facts concerning direct sums.}
If $(E_i)_{i\in I}$ is a family of locally
convex spaces,
we equip the direct sum $E:=\bigoplus_{i\in I}E_i$
with the locally convex direct sum topology~\cite{Bou}.
We identify $E_i$ with its canonical
image in~$E$.
\begin{rem}\label{semisum}
If $U_i\sub E_i$ is a $0$-neighbourhood for $i\in I$,
then the
convex hull
$U:=\conv\big(\bigcup_{i\in I} U_i\big)$
is a $0$-neighbourhood in~$E$,
and a basis of $0$-neighbourhoods is obtained in this way
(as is well-known).
If $I$ is countable, then the `boxes'
$\bigoplus_{i\in I}U_i:=E\cap \prod_{i\in I}U_i$ form
a basis of $0$-neighbourhoods in~$E$ (cf.\ \cite{Jar}).
It is clear from this that the topology on~$E$ is defined
by the seminorms $q\colon E\to [0,\infty[$ taking $x=(x_i)_{i\in I}$
to $\sum_{i\in I}q_i(x_i)$, for $q_i$ ranging through the set of continuous
seminorms on~$E_i$
(because $B^q_1(0)=\conv(\bigcup_{i\in I}B^{q_i}_1(0))$.)
If $I$ is countable, we can take the seminorms $q(x):=\max\{q_i(x_i)\colon i\in I\}$
instead (because $B^q_1(0)=\bigoplus_{i\in I}B^{q_i}_1(0))$.)
\end{rem}
\emph{Some types of locally convex spaces.}
If~$E$ is a topological vector space,
we write $E_{\lcx}$ for~$E$,
equipped with the finest among those
(not necessarily Hausdorff) locally convex vector topologies which are
coarser than the original topology (see, e.g., \cite{Glo}).
A topological space~$X$ is called a \emph{$k_\omega$-space}
if $X=\dl\,K_n$\vspace{-.3mm}
as a topological space,
for a sequence $K_1\sub K_2\sub\cdots$
of compact Hausdorff spaces
with continuous inclusion maps $K_n\to K_{n+1}$
(see~\cite{Fra}, \cite{GGH}).
We write $\R^{(\N)}$
for the space of finitely supported real sequences,
equipped with the finest locally convex vector topology.
Thus $\R^{(\N)}=\bigoplus_{n\in \N}\R$.\\[2.7mm]
\emph{Hypocontinuity.}
As a special case
of more general concepts,
we call a bilinear map $\beta\colon E_1\times E_2\to F$
between locally convex spaces
\emph{hypocontinuous} in its first argument (resp., in its second argument)
if it is separately continuous and
the restriction $\beta|_{B\times E_2}\colon B\times E_2\to F$
is continuous for each bounded subset $B\sub E_1$
(resp., $\beta|_{E_1\times B}$ is continuous
for each bounded subset $B\sub E_2$).
If $\beta$ is hypocontinuous in both arguments,
it will be called \emph{hypocontinuous}.\footnote{See,
e.g.,  \cite[Proposition~16.8]{Hyp} for the equivalence
of this definition with more classical ones
(cf.\ also Proposition~4 in
\cite[Chapter III, \S5, no.\,3]{Bou}).}
\section{Spaces with neighbourhood properties}\label{secinherit}
We recall basic examples of spaces
with the $\theta$-neighbourhood property,
and some permanence properties of
the class of such spaces.
\begin{prop}\label{excubc}
\begin{itemize}
\item[\rm(a)]
A metrizable locally convex space has the cnp if and only if it is normable.
Every normable space satisfies the $\theta$-np
for each infinite cardinal~$\theta$.
\item[\rm(b)]
Let $(E_n)_{n\in\N}$ be
a sequence of locally convex spaces that have the\linebreak
cnp.
Then also the locally convex direct sum $\bigoplus_{n\in\N}E_n$
has the cnp.
\item[\rm(c)]
Let $E$ be a locally convex space
that has the $\theta$-np
for some infinite cardinal~$\theta$.
Then also each vector subspace $F\sub E$ has
the $\theta$-np.
\item[\rm(d)]
Let $\theta$ be an infinite cardinal
and $E_1,\ldots, E_n$ be locally convex spaces that have
the $\theta$-np.
Then also $E_1\times\cdots\times E_n$
has the $\theta$-np.
\item[\rm(e)]
If a locally convex space~$E$ is a $k_\omega$-space
or $E=F_{\lcx}$ for a topological vector space~$F$ which is a $k_\omega$-space,
then $E$ has the cnp.
\item[\rm(f)]
$\R^{(\N)}$
has the cnp.
\item[\rm(g)]
For each infinite cardinal $\theta$,
there exists a locally convex space $E$
that has the
$\theta$-np but does not have the $\theta'$-np for any $\theta'>\theta$.
\item[\rm(h)]
If a locally convex space $E$ has the $\theta$-np for an infinite cardinal~$\theta$,
then also $E/F$ has the $\theta$-np, for every closed vector subspace
$F\sub E$.
\item[\rm(i)]
Let $E$ be the locally convex direct limit of a countable direct system of locally convex
spaces having the cnp. Then also $E$ has the cnp. In particular,
every LB-space has the cnp.
\item[\rm(j)] Every DF-space $($and every gDF-space$)$ has the cnp.
\end{itemize}
\end{prop}
\begin{proof}
(a) See \cite[1.1\,(i)]{Bon} and \cite[p.\,285]{Jor}.

(b) See \cite[p.\,223]{Fl2}.

(c) See \cite[p.\,285]{Jor}.

(d) Let $(p_j)_{j\in J}$ be
a family of continuous seminorms on $E:=E_1\times\cdots \times E_n$,
indexed by a set~$J$ of cardinality $|J|\leq \theta$.
Then there exist continuous seminorms~$p_{i,j}$ on~$E_i$
for $i\in\{1,\ldots, n\}$
with $p_j(x)\leq \max\{p_{1,j}(x_1),\ldots, p_{n,j}(x_n)\}$
for all $x=(x_1,\ldots, x_n)\in E$.
Since $E_i$ has the $\theta$-np,
there exists a continuous seminorm~$P_i$
on~$E_i$ such that
$P_{i,j}\preceq P_i$ for all $j\in J$,
and thus $P_{i,j}\leq C_{i,j} P_i$ with suitable $C_{i,j}>0$.
Then $p(x):=\max\{P_1(x_1),\ldots, P_n(x_n)\}$ defines a continuous seminorm~$p$
on~$E$ such that $p_j\preceq p$ for all~$j\in J$
(as $p_j\leq C_j p$ with $C_j:=\max\{C_{1,j},\ldots, C_{n,j}\}$).

(e) See \cite[Corollary~8.1]{Glo}.

(f) Since $\R^{(\N)}=\bigoplus_{n\in\N}\R$,
the assertion follows from~(a) and~(b).

(g) Let $X$ be a set of cardinality $|X|>\theta$
and $E:=\ell^\infty(X)$ be the space of all bounded real-valued functions on~$X$,
equipped with the (unusual!) topology defined by the seminorms
\[
\|.\|_Y\colon E\to [0,\infty[\,,\quad \gamma\mto \sup\{|\gamma(y)|\colon y\in Y\}\,,
\]
for $Y$ ranging through the subsets of~$X$ of cardinality
$|Y|\leq\theta$. It can be shown that $E$ has the asserted
properties (see \cite[Example~8.2]{Glo}).\footnote{That $E$ has the $\theta$-np
if $\theta=2^{\aleph_0}$ was also mentioned in \cite[p.\,285]{Jor}.}

(h) Let $\pi\colon E\to E/F$, $x\mto x+F$. If $J$ is a set of cardinality $\leq\theta$
and $(q_j)_{j\in J}$ a family of continuous seminorms on~$E/F$,
then the $q_j\circ\pi$ are continuous seminorms on~$E$, whence
there exists a continuous seminorm $p$ on $E$ and $c_j>0$ such that
$q_j\circ \pi\leq c_j p$ for all $j\in J$.
Let $q\colon E/F \to [0,\infty[$ be the Minkowski functional of $\pi(B^p_1(0))$.
Now $B^p_1(0)\sub c_j B^{q_j\circ \pi}_1(0)$
and thus also $B^p_1(0)+F\sub c_j B^{q_j\circ \pi}_1(0)$.
Hence $B^q_1(0)\sub c_j B^{q_j}_1(0)$
and thus $q_j\leq c_j q$.

(i) See \cite[p.\,223]{Fl2} for the first claim. With (a),
the final assertion follows.

(j) See \cite[Satz~1.1\,(i)]{Hol}.
\end{proof}
\section{Bilinear maps with product estimates}\label{pivot}
The folowing results provide
links between the cnp
and product estimates.
\begin{prop}\label{cubcuseful}
Let $E_1$, $E_2$ and $F$ be locally convex
spaces and
$\beta \!\!:$\linebreak
$E_1\times E_2\to F$ be a continuous bilinear map.
If~$E_1$ and~$E_2$ have the countable neighbourhood property,
then~$\beta$ satisfies product estimates.
\end{prop}
\begin{proof}
Let $p_{i,j}$ be continuous seminorms on~$F$ for $i,j\in\N$.
Since~$\beta$ is continuous bilinear, for each $(i,j)\in \N^2$
there exists a continuous seminorm~$P_{i,j}$ on~$E_1$
and a continuous seminorm~$Q_{i,j}$ on~$E_2$
such that $p_{i,j}(\beta(x,y))\leq P_{i,j}(x)Q_{i,j}(y)$ for all $(x,y)\in E_1\times E_2$.
Because~$E_1$ has the cnp,
there exists a continuous seminorm~$p$ on~$E_1$ such that
$P_{i,j}\preceq p$ for all $i,j\in\N$.
Likewise,
there exists a continuous seminorm~$q$ on~$E_2$ such that
$Q_{i,j}\preceq q$ for all $i,j\in\N$.
Thus, for $i,j\in \N$ there are $r_{i,j},s_{i,j}\in \,]0,\infty[$
such that $P_{i,j}\leq r_{i,j}p$ and $Q_{i,j}\leq s_{i,j}q$.
For $i\in\N$, let $a_i$ be the maximum of $1, r_{i,1}s_{i,1},\ldots, r_{i,i}s_{i,i}$,
and define $p_i:=a_ip$.
For $j\in\N$, let $b_j$ be the maximum of $1, r_{1,j}s_{1,j},\ldots, r_{j-1,j}s_{j-1,j}$,
and define $q_j:=b_j q$.
Let $i,j\in \N$. If $i\geq j$, then
\[
p_{i,j}(\beta(x,y))\leq P_{i,j}(x)Q_{i,j}(y)\leq r_{i,j}s_{i,j}p(x)q(y)
\leq a_ip(x)q(y)\leq p_i(x)q_j(y)
\]
for all $x\in E_1$ and $y\in E_2$.
If $i<j$, then $p_{i,j}(\beta(x,y))\leq r_{i,j}s_{i,j}p(x)q(y)
\leq b_jp(x)q(y)\leq p_i(x)q_j(y)$. Thus $\beta$ satisfies product estimates.
\end{proof}
Combining
Proposition~\ref{excubc}\,(a) and Proposition~\ref{cubcuseful}, we obtain:
\begin{cor}\label{examp1}
If $(E_1,\|.\|_1)$ and $(E_2,\|.\|_2)$
are normed spaces, then every continuous bilinear
map $\beta\colon E_1\times E_2\to F$ to a locally convex space $F$
admits product estimates.\,\Punkt
\end{cor}
\begin{cor}\label{Rinfty}
Every bilinear map from
$\R^{(\N)}\times \R^{(\N)}$ to a locally convex space
admits product estimates.
\end{cor}
\begin{proof}
It is well-known that
$\R^{(\N)}\times \R^{(\N)}
=\dl\, \R^n\times \R^n$\vspace{-.3mm}
as a topological space
(cf.\ \cite{Bis} and \cite[Theorem~4.1]{Hir}).
Since bilinear maps on $\R^n\times\R^n$
are always continuous, it follows that
every bilinear map $\beta$ from
$\R^{(\N)}\times \R^{(\N)}$
to a locally convex space is continuous.
Combining
Proposition~\ref{excubc}\,(f)
and Proposition~\ref{cubcuseful},
we deduce that~$\beta$ admits product estimates.
\end{proof}
\begin{rem}
The condition described in Proposition~\ref{cubcuseful}
is sufficient, but not necessary for product estimates.
For example, consider the convolution map
$\beta\colon C^\infty(K)\times C^\infty(K) \to C^\infty(K)$
on a non-discrete, compact Lie group~$K$.
Then $\beta$ satisfies product estimates
(by Theorem~C).
However, $C^\infty(K)$ is a non-normable,
metrizable space, and therefore does not have the cnp
(see Proposition~\ref{excubc}\,(a)).
\end{rem}
The next result was stimulated by a remark of C. Bargetz.\footnote{In a conversation
from May 11, 2012, C. Bargetz explained to the author that $\beta_b$ in Theorem~A
is continuous if $G=\R^n$, $r=s=t=\infty$ and $F$ is a quasicomplete DF-space,
as a consequence of a result on topological tensor products
by L. Schwartz and a result from his thesis~\cite{Bar}.
Since every DF-space has the cnp (cf.\ \cite[Satz~1.1(i)]{Hol}),
Proposition~\ref{newprop} shows that Bargetz' hypotheses
are subsumed by Theorem~A.}
\begin{prop}\label{newprop}
Let $E_1$, $E_2$ and $F$ be locally convex spaces.
If $F$ has the countable neighbourhood property,
then every continuous bilinear map $\beta\colon E_1\times E_2\to F$
admits product estimates.
\end{prop}
\begin{proof}
If $p_{i,j}$ are continuosu seminorms on~$F$ for
$i,j\in \N$, then the cnp of $F$ provides a continuous seminorm~$P$ on~$F$ and
real numbers $C_{i,j}>0$ such that $p_{i,j}\leq C_{i,j}P$ for all $i,j\in \N$.
Since $\beta$ is continuous, there exist continuous seminorms~$p$ on~$E_1$
and~$q$ on~$E_2$ such that $P(\beta(x,y))\leq p(x)q(y)$ for all $x\in E_1$ and $y\in E_2$.
By the lemma in~\cite{Bis}, there are $c_i>0$ for $i\in \N$ such that
$c_ic_j\leq 1/C_{i,j}$ for all $i,j\in \N$, and that $C_{i,j}\leq \frac{1}{c_ic_j}$.
Define $p_i:=\frac{1}{c_i}p$ and $q_j:=\frac{1}{c_j}q$.
Then $p_{i,j}(\beta(x,y))\leq C_{i,j}P(\beta(x,z))\leq C_{i,j}p(x)q(y)\leq \frac{1}{c_ic_j}p(x)q(y)\leq
p_i(x)q_j(y)$ for all $x\in E_1$ and $y\in E_2$, as required.
\end{proof}
For later use, let us record some obvious facts:
\begin{la}\label{laobvi}
Let $E_1$, $E_2$, $F$, $X_1$, $X_2$ and $Y$
be locally convex spaces,
$\beta:$\linebreak
$E_1\times E_2\to F$
be a continuous bilinear map
and $\lambda_1\colon X_1\to E_1$,
$\lambda_2 \colon X_2\to E_1$
and $\Lambda\colon F\to Y$ be continuous linear maps.
\begin{itemize}
\item[\rm(a)]
If $\beta$ admits product estimates,
then also $\Lambda\circ\beta$ and $\beta\circ (\lambda_1\times \lambda_2)$
admit product estimates.
\item[\rm(b)]
If $\Lambda$ is a topological embedding,
then $\beta$ admits product estimates if and only if $\Lambda\circ \beta$
admits product estimates.\,\Punkt
\end{itemize}
\end{la}
\section{Bilinear maps without product estimates}\label{secno}
We give two elementary examples of continuous bilinear maps
not admitting product estimates.
Further examples are provided by Theorems~B and~C.
\begin{example}\label{examp3}
We endow the direct power $A:=\R^\N$ with the product\linebreak
topology (of pointwise convergence),
which makes it a Fr\'{e}chet space and can be defined using the seminorms
\[
\|.\|_n\colon \R^\N\to [0,\infty[\,,\quad \|(x_i)_{i\in \N}\|_n:=\max\{|x_i|\colon 1\leq i\leq n\}
\]
for $n\in \N$.
Let $\beta\colon A\times A\to A$, $((x_i)_{i\in \N},(y_i)_{i\in \N})\mto (x_iy_i)_{i\in \N}$
be pointwise multiplication.
Then $\beta$ is a bilinear map and continuous, as $\|\beta(x,y)\|_n\leq \|x\|_n\|y\|_n$
for all $n\in \N$ and $x,y\in A$.
The map $\beta$ (which turns $A$ into a non-unital associative topological algebra)
does not satisfy product estimates.\\
To see this, consider the continuous seminorms $p_{i,j}:=\|.\|_{i+j}$ on~$A$.
Let $(p_i)_{i\in \N}$ and $(q_i)_{i\in \N}$
be any sequences of continuous seminorms on~$A$.
Then $p_1\leq r\|.\|_n$ for some $r>0$ and some $n\in \N$.
Let $e_{n+1}=(0,\ldots, 1,0,\ldots)\in A$ be the sequence with a single non-zero entry $1$
at position $n+1$.
Then
\[
p_{1,n}(\beta(e_{n+1},e_{n+1}))
=p_{1,n}(e_{n+1})
=\|e_{n+1}\|_{n+1}=1\,.
\]
However, $p_1(e_n)q_n(e_{n+1})\leq r\|e_{n+1}\|_nq_n(e_{n+1})=0q_n(e_{n+1})=0$.
Therefore $p_{1,n}(\beta(e_{n+1},e_{n+1}))>p_1(e_{n+1})q_n(e_{n+1})$,
and $\beta$ cannot have product estimates.
\end{example}
\begin{example}\label{examp4}
Consider the Fr\'{e}chet space $A:=C^\infty[0,1]:=C^\infty([0,1],\R)$,
whose vector topology is defined by the seminorms
\[
\|.\|_{C^k}\colon C^\infty[0,1]\to[0,\infty[\,,\quad \|\gamma\|_{C^k}:=\max\{\|\gamma^{(j)}\|_\infty\colon 0\leq j\leq k\}
\]
for $k\in \N_0$.
The Leibniz rule for derivatives of products implies that the bilinear pointwise multiplication map
$\beta\colon C^\infty[0,1]\times C^\infty[0,1]\to C^\infty[0,1]$, $\beta(\gamma,\eta):=\gamma\cdot\eta$
with $(\gamma\cdot\eta)(x):=\gamma(x)\eta(x)$
is continuous (since $\|\beta(\gamma,\eta)\|_{C^k}\leq 2^k\|\gamma\|_{C^k}\|\eta\|_{C^k}$), as is well-known.
We now show that $\beta$ does not satisfy product estimates.
To see this, let $p_{i,j}:=\|.\|_{C^{i+j}}$ for $i,j\in \N$.
Suppose that there exist continuous seminorms $p_i$ and $q_i$ on~$A$ for $i\in \N$, such that
\[
p_{i,j}(\beta(\gamma,\eta))\leq p_i(\gamma)q_j(\eta)\quad\mbox{for all $i,j\in \N$.}
\]
We derive a contradiction.
After increasing $p_1$, we may assume that $p_1=r\|.\|_{C^k}$ for some $r>0$ and some $k\in \N_0$.
Let $h\in A$ be a function whose restriction to $[\frac{1}{4},\frac{3}{4}]$ is identically~$1$.
For each $\gamma\in A$ with support $\Supp(\gamma)\sub [\frac{1}{4},\frac{3}{4}]$,
we then have
\[
\|\gamma\|_{C^{k+1}}=
\|\gamma\cdot h\|_{C^{k+1}}
=p_{1,k}(\gamma\cdot h)
\leq p_1(\gamma)q_k(h)
\leq K \|\gamma\|_{C^k}
\]
with $K:=r q_k(h)$.
Let $g\in C^\infty_c(\R)$ with $g(0)\not=0$
and $\Supp(g)\sub [{-\frac{1}{4}},\frac{1}{4}]$.
Then $g^{(j)}\not=0$ for all $j\in \N_0$
(because otherwise $g$ would be a polynomial and hence
not compactly supported, contradiction).
For $t\in \,]0,1]$, define $g_t\in A$ via
$g_t(x):=t^k g((x-\frac{1}{2})/t)$.
Then $g_t^{(j)}(x)=t^{k-j}g^{(j)}((x-\frac{1}{2})/t)$ for each $j\in \N_0$,
entailing that $S:=\sup\{\|g_t\|_{C^k}\colon t\in \,]0,1]\}<\infty$
and $\|g_t\|_{C^{k+1}}\geq \|g_t^{(k+1)}\|_\infty=t^{-1}\|g^{(k+1)}\|_\infty\to\infty$
as  $t\to 0$.
This contradicts the estimate
$\|g_t\|_{C^{k+1}}\leq K\|g_t\|_{C^k}\leq KS$.
\end{example}
\section{Spaces of vector-valued test functions}\label{sectest}
In this section, we compile preliminaries
concerning spaces of vector-valued test functions,
for later use. The proofs can be found in~\cite{BaG}.\\[2.7mm]
The manifolds considered in this article are finite-dimensional, smooth
and metrizable (but not necessarily $\sigma$-compact).\footnote{Recall that a manifold
is metrizable if and only if it is paracompact,
as follows, e.g.,  from \cite[Theorem II.4.1]{BaP}.}
The Lie groups considered are finite-dimensional, real Lie groups.\\[2.7mm]
\emph{Vector-valued $C^r$-maps on manifolds.}
If $r\in \N_0\cup\{\infty\}$,
$U\sub \R^n$ is open and $E$ a locally convex space,
then a map $\gamma\colon U\to E$ is called $C^r$ if the partial derivatives
$\partial^\alpha \gamma\colon U\to E$ exist and are continuous,
for all multi-indices $\alpha=(\alpha_1,\ldots,\alpha_n)\in \N_0^n$ such that $|\alpha|:=\alpha_1+\cdots+\alpha_n\leq r$.
If $V\sub \R^n$ is open and $\tau\colon V\to U$
a $C^r$-map, then also $\gamma\circ\tau$ is~$C^r$
(as a special case of infinite-dimensional
calculus as in
\cite{Mil}, \cite{Ham},
\cite{RES}, or \cite{GaN}).
It therefore makes sense to consider
$C^r$-maps from
manifolds to locally convex spaces.
If $M$ is a manifold and
$\gamma\colon M\to E$ a $C^1$-map to a locally convex space,
we write $d\gamma$ for the second component
of the tangent map $T\gamma\colon TM\to TE\isom E\times E$.
If $X\colon M\to TM$ is a smooth vector field on~$M$ and $\gamma$ as before,
we write
\[
X.\gamma:=d\gamma\circ X\,.
\]
\emph{The topology on $C^r_c(M,E)$.}
Let $r\in \N_0\cup\{\infty\}$ and $E$ be a locally convex space.
If $r=0$, let
$M$ be a (Hausdorff) locally compact space,
and equip the space $C^0(M,E):=C(M,E)$ of continuous $E$-valued
functions on~$M$
with the compact-open topology.
If $r>0$, let $M$ be a $C^r$-manifold.
Set $d^0\gamma:=\gamma$,
$T^0M:=M$,
$T^kM:=T(T^{k-1}M)$ and $d^k\gamma:=d(d^{k-1}\gamma)\colon T^kM\to E$
for $k\in\N$ with $k\leq r$.
Equip $C^r(M,E)$ with the initial topology with respect to the maps
$d^k\colon C^r(M,E)\to C(T^kM, E)$ for $k\in \N_0$ with $k\leq r$,
where $C(T^k(M),E)$ is equipped with the compact-open topology.
Returning to $r\in \N_0\cup\{\infty\}$,
endow the space
$C^r_K(M,E):=\{\gamma\in C^r(M,E)\colon \text{supp}(\gamma)\sub K\}$\linebreak
with the topology induced by $C^r(M,E)$,
for each compact subset~$K$\linebreak
of~$M$.
Let $\cK(M)$ be the set of compact subsets of~$M$.
Equip $C^r_c(M,E):=\bigcup_{K\in \cK(M)}\, C^r_K(M,E)$
with the locally convex direct limit topology.
Then $C^r_c(M,E)$ is Hausdorff (because the inclusion map
$C^r_c(M,E)\to C^r(M,E)$ is continuous).
As usual, we abbreviate $C^r(M):=C^r(M,\R)$,
$C^r_K(M):=C^r_K(M,\R)$ and
$C^r_c(M):=C^r_c(M,\R)$.
The following fact is well-known (see, e.g., \cite[Proposition~4.4]{GCX}):
\begin{la}\label{opa}
If $U\sub \R^n$ is open, $K\sub U$
compact and $r\in \N_0\cup\{\infty\}$, then the
topology on $C^r_K(U,E)$
arises from the seminorms $\|.\|_{k,p}$ defined via
\[
\|\gamma\|_{k,p}:=\max\{\|\partial^\alpha\gamma\|_{p,\infty}\colon \alpha\in \N_0^n,\,|\alpha|\leq k\},
\]
for all $k\in \N_0$ with $k\leq r$ and continuous seminorms $p$ on $E$.
\end{la}
In the next three lemmas
(which are Lemmas 1.3, 1.14 and 1.15 from \cite{BaG}, respectively),
we let $E$ be a locally convex space and $r\in \N_0\cup\{\infty\}$.
If $r=0$, we let $M$ be a locally compact space.
If $r>0$, then
$M$ is a manifold.
\begin{la}\label{lcsum}
Let $(h_j)_{j\in J}$ be a family of
functions $h_j\in C^r_c(M)$
whose\linebreak
supports $K_j:=\Supp(h_j)$
form a locally finite family.
Then the map
\[
\Phi\colon C^r_c(M,E)\to\bigoplus_{j\in J}C^r_{K_j}(M,E)
\,,\quad \gamma\mapsto (h_j\cdot \gamma)_{j\in J}
\]
is continuous and linear. If $(h_j)_{j\in J}$ is a partition of unity $($i.e., $h_j\geq 0$ and
$\sum_{j\in J}h_j=1$ pointwise$)$, then~$\Phi$ is a topological embedding.
\end{la}
\begin{la}\label{Phiv}
For each $0\not=v\in E$, the map
$\Phi_v\colon C^r_c(M)\to C^r_c(M,E)$,
$\Phi_v(\gamma):=\gamma v$
is linear and a topological embedding
$($where $(\gamma v)(x):=\gamma(x)v)$.
\end{la}
\begin{la}\label{scavect}
The map
$\Psi_{K,E}\colon C^r_K(M)\times E \to C^r_K(M,E)$, $(\gamma,v)\mto \gamma v$
is continuous,
for each compact subset $K\sub M$.
\end{la}
\begin{defn}\label{deflrmrtr}
Let $G$ be a Lie group, with identity element $1$,
and $K\sub G$ be a compact subset.
Let $\cB$ be a basis of the tangent space $T_1(G)$,
and $E$ be a locally convex space.
For $v\in\cB$, let $\cL_v$ be the left-invariant vector field
on~$G$ given by $\cL_v(g):=T_1(L_g)(c)$,
and $\cR_v$ the right-invariant vector field $\cR_v(g):=T_1(R_g)(v)$
(where $L_g, R_g\colon G\to G$, $L_g(x):=gx$, $R_g(x):=xg$).
Let
\[
\cF_L:=\{\cL_v\colon v\in \cB \}\quad\mbox{and}\quad \cF_R:=\{\cR_v \colon v\in \cB \}\,.
\]
Given $r\in \N_0\cup\{\infty\}$, $k,\ell\in\N_0$ with $k+\ell \leq r$, and a continuous seminorm~$p$ on~$E$,
we define $\|\gamma\|^L_{k,p}$ (resp., $\|\gamma\|^R_{k,p}$) for $\gamma\in C^r_K(G,E)$
as the maximum of the numbers
\[
\|X_j\ldots X_1.\gamma\|_{p,\infty}\,,
\]
for $j\in \{0,\ldots, k\}$ and $X_1,\ldots, X_j\in \cF_L$
(resp., $X_1,\ldots, X_j\in \cF_R$).
Define $\|\gamma\|^{R,L}_{k,\ell,p}$
as the maximum of the numbers
\[
\|X_i\ldots X_1.Y_j\ldots Y_1.\gamma\|_{p,\infty}\,,
\]
for $i\in \{0,\ldots, k\}$, $j\in \{0,\ldots, \ell\}$ and
$X_1,\ldots, X_i\in \cF_R$,
$Y_1,\ldots, Y_j\in \cF_L$.
Then $\|.\|^L_{k,p}$, $\|.\|^R_{k,p}$ and $\|.\|^{R,L}_{k,\ell,p}$
are seminorms on $C_K^r(G,E)$.
If $E=\R$ and $p=|.|$,
we relax notation and write
$\|.\|^R_k$ instead of
$\|.\|^R_{k,p}$.
\end{defn}
In the situation of Definition~\ref{deflrmrtr}, we have
the following (see \cite[Lemma~1.8]{BaG}):
\begin{la}\label{enoughsn}
For each $t\in\N_0\cup\{\infty\}$, compact set $K\sub G$
and locally convex space~$E$, the topology on $C^t_K(G,E)$
coincides with the topologies defined by each of the following families
of seminorms:
\begin{itemize}
\item[\rm (a)]
The family of the seminorms $\|.\|^L_{j,p}$, for $j\in \N_0$ such that $j\leq t$
and continuous seminorms~$p$ on~$E$;
\item[\rm (b)]
The family of the seminorms $\|.\|^R_{j,p}$, for $j\in \N_0$ such that $j\leq t$
and continuous seminorms~$p$ on~$E$.
\end{itemize}
If $t<\infty$ and $t=k+\ell$, then the topology on $C^t_K(G,E)$
is also defined by the seminorms $\|.\|^{R,L}_{k,\ell,p}$, for continuous
seminorms $p$ on $E$.
\end{la}
To enable uniform notation in the proofs for Lie groups and locally compact groups,
we write $\|.\|^L_{0,p}:=\|.\|^R_{0,p}:=\|.\|^{R,L}_{0,0,p}:=\|.\|_{p,\infty}$
if $p$ is a continuous seminorm on~$E$ and $G$ a locally compact
group. We also write $\|.\|^R_0:=\|.\|_\infty$.
For example, Lemma~\ref{enoughsn}
then remains valid for locally compact groups~$G$.\\[3mm]
The following fact (covered by \cite[Lemma~2.6]{BaG}) will be used repeatedly:
\begin{la}\label{lasimpe}
Let $(G,r,s,t,b)$ be as in \emph{\ref{setting}}, $K\sub G$ be compact,
$\gamma\in C^r_K(G,E_1)$, $\eta\in C^s_c(G,E_2)$
and $q$, $p_1$, $p_2$ be continuous seminorms on $F$, $E_1$
and $E_2$,\linebreak
respectively, such that $q(b(x,y))\leq p_1(x)p_2(y)$
for all $(x,y)\in E_1\times E_2$. Let $k,\ell\in \N_0$
with $k\leq r$ and $\ell\leq s$. Then
\[
\|\gamma*_b\eta\|^{R,L}_{k,\ell,q} \, \leq \,  \|\gamma\|_{k,p_1}^R\|\eta\|_{\ell,p_2}^L\lambda_G(K).
\]
\end{la}
\section{Proof of Theorem~B}\label{secscavect}
First,
we briefly discuss
the compact covering number.
\begin{la}\label{lacov}
Let $M$ be a paracompact,
locally compact, non-compact\linebreak
topological space.
Then the following holds:
\begin{itemize}
\item[\rm(a)]
$M$ is $\sigma$-compact if and only if $\theta(M)=\aleph_0$.
\item[\rm(b)]
$\theta(M)=|J|$
for every
locally finite cover $(V_j)_{j\in J}$ of~$M$
by relatively\linebreak
compact, open,
non-empty sets.
\item[\rm(c)]
$M$ can be expressed as a topological sum $($disjoint union$)$
of open, $\sigma$-compact, non-empty subsets $U_j$, $j\in J$.
Then $\theta(M)=\max\{|J|,\aleph_0\}$.
\item[\rm(d)]
If $M$ is a manifold, then $\theta(M)$
is the maximum of $\aleph_0$
and the number of connected components
of~$M$.
\end{itemize}
\end{la}
\begin{proof}
(a) By definition, $M$ is $\sigma$-compact if and only if $\theta(M)\leq\aleph_0$;
and as $M$ is assumed non-compact,
this is equivalent to $\theta(M)=\aleph_0$.

(b) We have $\theta(M)\leq |J|$ by minimality,
as $(\wb{V_j})_{j\in J}$
is a compact cover.
For the converse, let
$(K_a)_{a\in A}$
be a cover of $M$ by compact sets, with $|A|=\theta(M)$.
Then $J_a:=\{j\in J\colon K_a\cap V_j\not=\emptyset\}$
is finite, for each $a\in A$.
Hence $|J|=|\bigcup_{a\in A}J_a|\leq |A|\aleph_0=|A|=\theta(M)$
and thus $|J|=\theta(M)$.

(c) The first assertion is well known~\cite{Eng}.
Each $U_j$ admits a
countable, locally finite cover $(V_{j,i})_{i\in I_j}$
by relatively compact, open, non-empty sets. Let $L:=\coprod_{j\in J}I_j$
be the disjoint union of the sets $I_j$.
Then $(V_{j,i})_{(j,i)\in L}$
is a locally finite, relatively compact open cover
of $M$.
Moreover, $J$ or one of the sets $I_j$ is infinite.
Hence $\theta(M)=|L|=\max\{|J|,\aleph_0\}$.

(d) Apply (c) to the partition of~$M$ into its connected
components.
\end{proof}
{\bf Proof of Theorem~B.}
If $\gamma\in C^r_c(M)$, let $K:=\Supp(\gamma)$.
Because $\Psi_{K,E}$ from Lemma~\ref{scavect}
is continuous, also
$\Psi_{c,E}(\gamma,.)=\Psi_{K,E}(\gamma,.)$ is continuous.
For each $v\in E$, the linear
map $\Psi_{c,E}(.,v)=\Phi_v$ is continuous, by Lemma~\ref{Phiv}.
Hence $\Psi_{c,E}$ is separately continuous.
As is clear, $\Psi_{c,E}$ is bilinear.
For each bounded set $B\sub C^r_c(M)$, there exists a compact set $K\sub M$
such that $B\sub C^r_K(M)$
(see, e.g., \cite[Lemma~1.16\,(c)]{BaG}).
Hence $\Psi_{c,E}|_{B\times E}=\Psi_{K,E}|_{B\times E}$
is continuous
and thus $\Psi_{c,E}$ is hypocontinuous in the first argument.
Since each $C^r_K(M)$ is a Fr\'{e}chet space and hence barrelled,
$C^r_c(M)=\dl\, C^r_K(M)$\vspace{-.3mm}
is a locally convex direct limit
of barrelled spaces and hence barrelled \cite[II.7.2]{SaW}.
The separately continuous bilinear map $\Psi_{c,E}$ on $C^r_c(M)\times E$
is therefore hypocontinuous in the second argument \cite[III.5.2]{SaW}.
Hence $\Psi_{c,E}$ is hypocontinuous.

We let $(U_j)_{j\in J}$ be a locally finite cover of $M$
by relatively compact, open sets~$U_j$. Then $|J|=\theta(M)$
(see Lemma~\ref{lacov}\,(b)).
Let $(h_j)_{j\in J}$ be a $C^r$-partition of unity subordinate to $(U_j)_{j\in J}$,
in the sense that $K_j:=\Supp(h_j)\sub U_j$.
Then also those $U_j$ with $h_j\not=0$ form a cover.
We may therefore assume that $h_j\not=0$ for all $j\in J$.

Now suppose that $\Psi_{c,E}$ is continuous.
Let $p_j$ be a continuous seminorm on~$E$, for each $j\in J$.
Let $U$
be the set of all $\gamma\in C^r_c(M,E)$
such that $\|h_j\gamma\|_{p_j,\infty}\leq 1$ for all $j\in J$.
Because
\[
\Phi\colon C^r_c(M,E)\to\bigoplus_{j\in J}C^r_{K_j}(M,E), \quad
\gamma\mto (h_j\gamma)_{j\in J}
\]
is continuous (see Lemma~\ref{lcsum}),
$U$ is a $0$-neighbourhood.
Hence, there are $0$-neighbourhoods $V\sub C^r_c(M)$ and $W\sub E$
such that $\Psi_{c,E}(V\times W)$\linebreak
$\sub U$.
After shrinking $W$, we may assume that $W=\wb{B}^q_1(0)$
for some continuous seminorm~$q$ on~$E$.
For each $j\in J$, we have $\ve_j h_j \in V$ for some $\ve_j>0$.
Hence $\Psi_{c,E}(\ve_j h_j,w)\in U$ for each $w\in W$
and thus $1\geq \|\ve_j h_j w\|_{p_j,\infty}=\ve_j p_j(w)\|h_j\|_\infty$.
So, abbreviating $C_j:=1/(\ve_j\|h_j\|_\infty)$,
we have $p_j(w)\leq C_j$ for all $w\in \wb{B}^q_1(0)$
and thus $p_j\leq C_j q$. Hence $p_j\preceq q$
for all $j$ and thus $E$ has the $\\theta(M)$-np.

Conversely, let $E$ have the $\theta(M)$-np.
If $r\geq 1$, we can cover each $\wb{U_j}$ with finitely many chart domains
$W_{j,i}$ and replace $U_j$ by $U_j\cap W_{j,i}$,
without increasing the cardinality of the family (since $|J|\aleph_0=|J|$).
We may\linebreak
therefore assume that each $U_j$ is the domain of a chart
$\phi_j\colon U_j\to V_j\sub \R^n$. 
Let $U\sub C^r_c(M,E)$ be a $0$-neighbourhood.
Because $\Phi$ just defined is a\linebreak
topological embedding,
after shrinking $U$ we may assume that there are\linebreak
continuous seminorms
$p_j$ on~$E$ and $k_j\in \N_0$ such that $k_j\leq r$ and
\[
U=\{\gamma\in C^r_c(M,E)\colon \sum_{j\in J}\|(h_j \gamma)\circ \phi_j^{-1}\|_{k_j,p_j}<1\}
\]
(see Lemma~\ref{opa}).
By the $\theta(M)$-np,
there exists a continuous seminorm $q$ on~$E$
and a family $(C_j)_{j\in J}$ of real numbers $C_j>0$ such
that $p_j\leq C_j q$ for each $j\in J$.
Then
\[
V:=\{\gamma\in C^r_c(M)\colon \sum_{j\in J}C_j\|(h_j \gamma)\circ \phi_j^{-1}\|_{k_j}<1\}
\]
is a $0$-neighbourhood in $C^r_c(M)$ and $\Theta_{c,E}(V\!\times \!\wb{B}^q_1(0))\!\sub \!U$
as $\|(h_j \gamma w)\circ\phi_j^{-1}\|_{k_j,p_j}$\linebreak
$=p_j(v)\|(h_j \gamma)\circ \phi_j^{-1}\|_{k_j}
\leq C_j q(v)\|(h_j \gamma)\circ \phi_j^{-1}\|_{k_j}\leq C_j \|(h_j \gamma)\circ \phi_j^{-1}\|_{k_j}$,
with sum~$<1$.
Hence $\Theta_{c,E}$ is continuous at~$(0,0)$ and hence continuous.

If~$E$ is normable, then $E$ has the $\theta(M)$-np (see
Proposition~\ref{excubc}\,(a)), whence $\Psi_{c,E}$ is continuous.
If~$E$ is metrizable and $\Psi_{c,E}$ is continuous,
then~$E$ has the $\theta(M)$-np,
and thus~$E$ has the cnp.
Hence~$E$ is normable (by Proposition~\ref{excubc}\,(a)).
\section{Proof of Theorem~C}\label{secconvol}
\begin{la}\label{preforint}
Let $(G,r,s,t,b)$ and $\beta_b$
be as in \emph{\ref{setting}}.
If $\beta_b$ admits product\linebreak
estimates, then also $b$ admits product estimates.
\end{la}
\begin{proof}
Let $K\sub G$ be a compact identity neighbourhood.
If the map $\beta_b\colon C^r_c(G,E_1)\times C^s_c(G,E_2)\to C^t_c(G,F)$
admits product estimates,
then also the convolution map
$\theta\colon C^r_K(G,E_1)\times C^s_K(G,E_2)\to C^t_{KK}(G,F)$
admits\linebreak
product estimates,
being obtained via restriction and co-restriction from~$\beta_b$
(see Lemma~\ref{laobvi} (a) and~(b)).\\[2.4mm]
If $G$ is discrete, we simply take $K:=\{1\}$,
in which case~$b$ can be identified with~$\theta$
and hence admits product estimates -- as required.\\[2.4mm]
For general~$G$, choose a non-zero function
$h\in C^0_K(G)$ such that $h\geq 0$;
if $G$ is a Lie group, we assume that~$h$ is smooth.
After replacing $h$ with the function $y\mto h(y)+h(y^{-1})$
if necessary, we may assume that $h(y)=h(y^{-1})$ for all $y\in G$.
Also, after replacing $h$ with a positive multiple if necessary,
we may assume that $\int_Gh(y)^2\,d\lambda_G(y)=1$.
Then $\phi_1\colon E_1\to C^r_K(G,E_1)$, $u\mto h u$
and  $\phi_2\colon E_2\to C^r_K(G,E_2)$, $v\mto h v$
are continuous linear maps,
and also $\ve\colon C^t_{KK}(G,F)\to F$, $\gamma\mto \gamma(1)$
is continuous linear.
Hence
$\ve\circ \theta\circ (\phi_1\times \phi_2):$\linebreak
$E_1\times E_2\to F$
admits product estimates, by Lemma~\ref{laobvi}\,(a).
But this map takes $(u,v)\in E_1\times E_2$ to
\[
(hu * hv)(1)=b(u,v)\int_Gh(y)h(y^{-1})\,d\lambda_G(y)
=b(u,v)
\]
and thus coincides with~$b$. Hence $b$ admits product estimates.
\end{proof}
The next lemma, as well as Lemmas \ref{counteg} and \ref{viaquot}, are relevant only for the study of convolution
on Lie groups. Readers exclusively interested
in the case that~$G$ is a locally compact group and $r=s=t=0$
can skip them.
\begin{la}\label{examp2}
Let $(G,r,s,t,b)$ and $\beta_b$ be as in \emph{\ref{setting}}. 
If $\beta_b$ admits product\linebreak
estimates, $t=\infty$
and $G$ is not discrete,
then also $r=s=\infty$.
\end{la}
\begin{proof}
Because $\beta:=\beta_b$ admits product estimates,
it is continuous.
Hence, if $G$ is not compact, then $r=s=\infty$ by Theorem~A.
It remains to show that $\beta$
does not admit product estimates if~$G$ is compact,
$r<\infty$, and $s=\infty$ (the case $r=\infty$, $s<\infty$
can be settled along similar lines).
As a tool, let $\theta\colon C^r(G)\times C^s(G)\to C^t(G)$
be the convolution of scalar-valued functions.
Pick $u\in E_1$, $v\in E_2$ such that $w:=b(u,v)\not=0$.
Let $\Phi_u\colon C^r(G)\to C^r(G,E_1)$,
$\Phi_v\colon C^s(G)\to C^s(G,E_2)$
and $\Phi_w\colon C^t(G)\to C^t(G,F)$
be the topological embeddings taking $\gamma$
to $\gamma u$, $\gamma v$ and $\gamma w$, respectively
(see Lemma~\ref{Phiv}).
In view of Lemma~\ref{laobvi}\,(b),
if we can show that~$\theta$ does not admit product
estimates,
then $\Phi_w\circ\theta=\beta\circ (\Phi_u\times \Phi_v)$
will not admit product estimates either
(Lemma~\ref{laobvi}\,(b)).
Hence also $\beta$ does not admit product estimates
(Lemma~\ref{laobvi}\,(a)).
We may therefore assume that $E_1=E_2=F=\R$
and $\beta=\theta$.
Consider the continuous\linebreak
seminorms
$P_{i,j}:=\|.\|^L_i$ on $C^\infty(G)$ for $i,j\in \N$.
If~$\beta$ would admit product estimates (which will lead
to a contradiction), we could find continuous seminorms
$P_i$ on $C^r(G)$ and $Q_j$ on $C^\infty(G)$ such that
$P_{i,j}(\gamma*\eta)\leq P_i(\gamma)Q_j(\eta)$.
After increasing $P_i$ and $Q_j$
if necessary, we may assume that $P_i=a_i\|.\|^L_r$
and $Q_j=c_j\|.\|^L_{s_j}$ with suitable $a_i,c_j>0$
and $s_j\in \N_0$ (see Lemma~\ref{enoughsn}).
Thus
\[
\|\gamma*\eta\|^L_i\leq a_ic_j\|\gamma\|^L_r\|\eta\|^L_{s_j}
\]
for all $i,j\in \N$ and $(\gamma,\eta)\in C^r(G)\times C^\infty(G)$.
In particular, with $j:=1$ and $\ell:=s_1\in \N_0$, we obtain
\[
\|\gamma*\eta\|^L_i\leq a_ic_1\|\gamma\|^L_r\|\eta\|^L_\ell
\]
for all $i\in \N$ and $(\gamma,\eta)\in C^r(G)\times C^\infty(G)$.
Hence $\beta_b$ would be continuous
as a map $(C^r(G),\|.\|^L_r)\times (C^\infty(G),\|.\|_\ell^L)\to C^\infty(G)$,
using the usual Fr\'{e}chet topology on the right hand side,
but only the indicated norms on the left. This is impossible,
as recorded in \cite[Lemma~5.1]{BaG}.
\end{proof}
\begin{la}\label{examp2b}
Let $(G,r,s,t,b)$ and $\beta_b$ be as in \emph{\ref{setting}}.
Assume that $G$ is $\sigma$-compact
and $b$ admits product estimates.
Moreover, assume that $t<\infty$ or $r=s=t=\infty$.
Then also $\beta_b$ admits
product estimates.
\end{la}
The proof will be based on three lemmas:
\begin{la}\label{specifsum}
Let $(E_i)_{i\in \N}$ and $(F_j)_{j\in \N}$
be sequences of locally convex spaces, $H$ be a locally
convex space and $\beta_{i,j}\colon E_i\times F_j\to H$
be continuous bilinear maps for $i,j\in\N$.
Assume that, for every double sequence $(P_{\sigma,\tau})_{\sigma,\tau\in \N}$
of continuous seminorms on~$H$,
there are continuous seminorms~$P_{i,\sigma}$ on $E_i$,
continuous seminorms~$Q_{j,\tau}$ on $F_j$
and numbers $C_{i,j,\sigma,\tau}>0$, such that
\[
P_{\sigma,\tau}(\beta_{i,j}(x,y))\leq C_{i,j,\sigma,\tau}P_{i,\sigma}(x)Q_{j,\tau}(y)
\]
for all $i,j,\sigma,\tau\in \N$ and all $x\in E_i$ and $y\in F_j$.
Then the bilinear map
$\beta\colon \!\!\big(\bigoplus_{i\in\N}E_i\big)\times
\big(\bigoplus_{j\in\N}F_j\big)\!\to H$
taking $((x_i)_{i\in \N},(y_j)_{j\in\N})$ to $\sum_{i,j\in \N}\beta_{i,j}(x_i,y_j)$
admits product estimates.
\end{la}
\begin{proof}
The map $b\colon \R^{(\N)}\times \R^{(\N)}\to \R^{(\N\times\N)}$,
$b((u_i)_{i\in \N}, (v_j)_{j\in \N}):=(u_iv_j)_{(i,j)\in \N\times\N}$
admits product estimates, by Corollary~\ref{Rinfty}.
For all $\sigma,\tau\in \N$,
\[p_{\sigma,\tau}(w):=\sum_{i,j\in \N}
C_{i,j,\sigma,\tau}|w_{i,j}|
\]
defines a continuous seminorm on $\R^{(\N\times\N)}$
(see Remark~\ref{semisum}).
Hence, there exist continuous seminorms
$p_\sigma$ and $q_\tau$ on $\R^{(\N)}$ such that $p_{\sigma,\tau}(b(u,v))\leq p_\sigma(u)q_\tau(v)$
for all $u,v\in \R^{(\N)}$.
By Remark~\ref{semisum},
after increasing $p_\sigma$ and $q_\tau$
if necessary, we may assume that they are of the form
\[
p_\sigma(u)=\max\{r_{i,\sigma}|u_i|\colon i\in \N\}
\]
and $q_\tau(v)=\max\{s_{j,\tau} |v_j|\colon j\in \N\}$
with suitable $r_{i,\sigma},s_{j,\tau}>0$.
Then
\[
P_\sigma(x) :=p_\sigma((P_{i,\sigma}(x_i))_{i\in \N})=\max\{r_{i,\sigma}P_{i,\sigma}(x_i)\colon i\in\N\}
\]
and $Q_\tau(y):=q_\tau((Q_{j,\tau}(y_j))_{j\in\N})$
(for $x\in E:=\bigoplus_{i\in \N}E_i$, $y\in F:=\bigoplus_{j\in \N}F_j$)
define continuous seminorms $P_\sigma$ and $Q_\tau$
on~$E$ and~$F$, respectively (see\linebreak
Remark~\ref{semisum}).
For all $\sigma,\tau\in \N$ and $x,y$ as before, we obtain
\begin{eqnarray*}
P_{\sigma,\tau}(\beta(x,y)) & \leq & \sum_{i,j\in \N}P_{\sigma,\tau}(\beta_{i,j}(x_i,y_j))
\,\leq\,  \sum_{i,j\in \N}C_{i,j,\sigma,\tau}P_{i,\sigma}(x_{i,\tau})Q_{j,\tau}(y_j)\\
&=& p_{\sigma,\tau}(b((P_{i,\sigma}(x_i))_{i\in \N},(Q_{j,\tau}(y_j))_{j\in \N}))\\
&\leq& p_\sigma((P_{i,\sigma}(x_i))_{i\in\N})q_\tau((Q_{j,\tau}(y_j))_{j\in \N})
\,=\,P_\sigma(x)Q_\tau(y)\,.
\end{eqnarray*}
Hence $\beta$ admits product estimates.
\end{proof}
\begin{la}\label{counteg}
Let $A$ be a countable set and $t_{\alpha,\beta}\in \N_0$
for $\alpha,\beta \in A$.
Then there exist $r_\alpha,s_\beta\in \N_0$
for $\alpha,\beta\in A$ such that
\[
(\forall \alpha,\beta\in A) \;\; r_\alpha+s_\beta\, \geq\,  t_{\alpha,\beta}\,.
\]
\end{la}
\begin{proof}
If $A$ is a finite set, the assertion is trivial.
If $A$ is infinite, we may assume that $A=\N$.
For $i\in \N$, let $r_i:=\max\{t_{i,j}\colon j\leq i\}$.
For $j\in \N$, let $s_j:=\max\{t_{i,j}\colon i\leq j\}$.
If $i,j\in \N$ and $i<j$, we deduce $t_{i,j}\leq s_j\leq r_i+s_j$.
Likewise, $t_{i,j}\leq r_i\leq r_i+s_j$ if $i\geq j$.
\end{proof}
\begin{la}\label{viaquot}
Let $G$ be a Lie group, $E$ be a locally convex space,
$K\sub G$ be compact,
$p$ be a continuous seminorm on~$E$
and $k,\ell\in \N_0$.
Then there exists $C>0$ such that
$\|\gamma\|^L_{k+\ell,p}\leq C\|\gamma\|^{R,L}_{k,\ell,p}
=\|\gamma\|^{R,L}_{k,\ell,C\cdot p}$ for all $\gamma\in C^{k+\ell}_K(G,E)$.
\end{la}
\begin{proof}
Let $E_p=E/p^{-1}(0)$ be the corresponding normed space,
$\pi\colon E\to E_p$ be the canonical
map and $P:=\|.\|_p$ be the norm on~$E_p$.
Because both $\|.\|^L_{k+\ell,P}$ and $\|.\|^{R,L}_{k,\ell,P}$
define the topology of $C^{k+\ell}_K(G,E_p)$
(see Lemma~\ref{enoughsn}),
there exists $C>0$ such that $\|.\|^L_{k+\ell,P}\leq C\|.\|^{R,L}_{k,\ell,P}$.
Thus $\|\gamma\|^L_{k+\ell,p}=$\linebreak
$\|\pi\circ \gamma\|^L_{k+\ell,P}\leq C\|\pi\circ\gamma\|^{R,L}_{k,\ell,P}
=C\|\gamma\|^{R,L}_{k,\ell,p}$ for all $\gamma\in C^{k+\ell}_K(G,E)$.
\end{proof}
{\bf Proof of Lemma~\ref{examp2b}.}
Let $(h_i)_{i\in\N}$ be a partition
of unity for $G$ (smooth if $G$ is a Lie group,
continuous if $G$ is merely a locally compact group).
Let $\Phi\colon C^r_c(G,E_1)\to\bigoplus_{i\in\N} C^r_{K_i}(G,E_1)$
and $\Psi\colon C^s_c(G,E_2)\to\bigoplus_{i\in\N} C^s_{K_i}(G,E_2)$
be the embeddings taking $\gamma$ to $(h_i\gamma)_{i\in \N}$
(see Lemma~\ref{lcsum}).
Let
\[
f\colon \bigoplus_{i\in \N}C^r_{K_i}(G,E_1)\times \bigoplus_{j\in \N}C^s_{K_j}(G,E_2)\to C^t_c(G,F)
\]
be the map taking $((\gamma_i)_{i\in\N},(\eta_j)_{j\in \N})$
to $\sum_{i,j\in \N}\gamma_i*_b\eta_j$.
Since $\beta_b=f\circ (\Phi\times \Psi)$, we need only show
that~$f$ admits product estimates (Lemma~\ref{laobvi}).
To verify the latter property, let $P_{\sigma,\tau}$ be continuous
seminorms on $C^t_c(G,F)$ for $\sigma,\tau\!\in \!\N$.\\[2.4mm]
Before we turn to the general case, let us consider the instructive special case $r=s=t=0$
(whose proof is much simpler).
For all $i,j,\sigma,\tau\in \N$,
there exists a continuous seminorm $P_{i,j,\sigma,\tau}$ on~$F$
such that
\[
P_{\sigma,\tau}(\gamma)\leq \|\gamma\|_{P_{i,j,\sigma,\tau},\infty}
\]
for all $\gamma\in C_{K_iK_j}(G,F)$ (cf.\ Lemma~\ref{enoughsn}
and the lines thereafter).
Since
$b\colon E_1\times E_2\to F$ admits product estimates
and the set $\N\times \N$ (which contains the $(i,\sigma)$ and $(j,\tau)$)
admits a bijective map $\N\times \N\to\N$,
there exist continuous seminorms
$p_{i,\sigma}$ on~$E_1$ and $q_{j,\tau}$ on~$E_2$
such that
\[
P_{i,j,\sigma,\tau}(b(x,y))\leq p_{i,\sigma}(x)q_{j,\tau}(y)
\]
for all $i,j,\sigma,\tau\in \N$ and $x\in E_1$, $y\in E_2$.
Define $S_{i,\sigma} \colon C_{K_i}(G,E_1)\to[0,\infty[$
and $Q_{j,\tau}\colon C_{K_j}(G,E_2)\to[0,\infty[$
via $S_{i,\sigma}:=\lambda_G(K_i)\|.\|_{p_{i,\sigma},\infty}$ and
$Q_{j,\tau}:=\|.\|_{q_{j,\tau},\infty}$,
respectively.
Then
\[
P_{\sigma,\tau}(\gamma*_b\eta) \leq  \|\gamma*_b\eta\|_{P_{i,j,\sigma,\tau},\infty}
\leq \|\gamma\|_{p_{i,\sigma},\infty}\|\eta\|_{q_{j,\tau},\infty}\lambda_G(K_i)
= S_{i,\sigma}(\gamma)Q_{j,\tau}(\eta)\]
for all $(\gamma,\eta)\in C_{K_i}(G,E_1)\times C_{K_j}(G,E_2)$
(using Lemma~\ref{lasimpe} for the second inequality).
The hypotheses of Lemma~\ref{specifsum} are therefore satisfied,
whence~$f$ (and hence also~$\beta_b$)
admits product estimates.\\[2.4mm]
We now complete the proof of the lemma in full generality.
In the case $t<\infty$, we
choose $k,\ell\in \N_0$
with $k\leq r$, $\ell\leq s$ and $k+\ell=t$.
For all $i,j\in \N$,
there exists a continuous seminorm $P_{i,j,\sigma,\tau}$ on~$F$
such that
\[
P_{\sigma,\tau}(\gamma)\leq \|\gamma\|^{R,L}_{k,\ell,P_{i,j,\sigma,\tau}}
\]
for $\gamma\in C^t_{K_iK_j}(G,F)$ (Lemma~\ref{enoughsn}).
Set $r_{i,\sigma}:=k$ and $s_{j,\tau}:=\ell$ for $i,j,\sigma,\tau\in \N$.

In the case $r=s=t=\infty$,
there exist $t_{i,j,\sigma,\tau}\in \N_0$
and continuous seminorms
$Q_{i,j,\sigma,\tau}$ on~$F$
such that
$P_{\sigma,\tau}(\gamma)\leq \|\gamma\|^L_{t_{i,j,\sigma,\tau},Q_{i,j,\sigma,\tau}}$
for all $\gamma\in C^\infty_{K_iK_j}(G,F)$.
Using Lemma~\ref{counteg},
we find $r_{i,\sigma},s_{j,\tau}\in \N_0$
such that
\[
r_{i,\sigma}+s_{j,\tau} \geq t_{i,j,\sigma,\tau}
\]
for all $i,j,\sigma,\tau\in \N$.
Then
$\|.\|^L_{t_{i,j,\sigma,\tau},Q_{i,j,\sigma,\tau}}\leq
\|.\|^L_{r_{i,\sigma}+s_{j,\tau},Q_{i,j,\sigma,\tau}}\leq
\|.\|^{R,L}_{r_{i,\sigma},s_{j,\tau},P_{i,j,\sigma,\tau}}$
on $C^\infty_{K_i,K_j}(G,F)$,
with some positive multiple $P_{i,j,\sigma,\tau}$ of
$Q_{i,j,\sigma,\tau}$ (Lemma~\ref{viaquot}).

In either case, since
$b\colon E_1\times E_2\to F$ admits product estimates,
there exist continuous seminorms
$p_{i,\sigma}$ on~$E_1$ and $q_{j,\tau}$ on~$E_2$
such that
\[
P_{i,j,\sigma,\tau}(b(x,y))\leq p_{i,\sigma}(x)q_{j,\tau}(y)
\]
for all $i,j,\sigma,\tau\in \N$ and $x\in E_1$, $y\in E_2$.
Define $S_{i,\sigma} \colon C^r_{K_i}(G,E_1)\to[0,\infty[$
and $Q_{j,\tau}\colon C^s_{K_j}(G,E_2)\to[0,\infty[$
via $S_{i,\sigma}:=\lambda_G(K_i)\|.\|^R_{r_{i,\sigma},p_{i,\sigma}}$ and $Q_{j,\tau}:=\|.\|^L_{s_{j,\tau},q_{j,\tau}}$,
respectively.
Then
\begin{eqnarray*}
P_{\sigma,\tau}(\gamma*_b\eta) &\leq  &\|\gamma*_b\eta\|^{R,L}_{r_{i,\sigma},s_{j,\tau},P_{i,j,\sigma,\tau}}
\leq \|\gamma\|^R_{r_{i,\sigma},p_{i,\sigma}}\|\eta\|^L_{s_{j,\tau},q_{j,\tau}}\lambda_G(K_i)\\
&= & S_{i,\sigma}(\gamma)Q_{j,\tau}(\eta)
\end{eqnarray*}
for all $(\gamma,\eta)\in C^r_{K_i}(G,E_1)\times C^s_{K_j}(G,E_2)$
(using Lemma~\ref{lasimpe}).
As the hypotheses of Lemma~\ref{specifsum} are satisfied,
$f$ (and thus~$\beta_b$)
admits product estimates.\vspace{3mm}\,\Punkt

\noindent
{\bf Proof of Theorem~C.}
\emph{Case}~1: \emph{$G$ is a finite group}.
Then $G$ is compact and hence~$\beta_b$ is always continuous
\cite[Corollary~2.3]{BaG}.
If $\beta_b$ admits product estimates, then also $b$ admits these
(Lemma~\ref{preforint}).
If $b$ admits product estimates,
then $\beta_b$
admits product estimates, by Lemma~\ref{examp2b}
(note that any $(r,s,t)$ can be replaced with $(0,0,0)$
without changing the function spaces).

\emph{Case}~2: \emph{$G$ is an infinite discrete group}.
If $\beta_b$ is continuous,
then $G$ is countable and~$b$ admits product estimates,
by \cite[Proposition~6.1]{BaG}.
If $G$ is countable and $b$ admits product estimates,
then~$\beta_b$
admits product estimates, by Lemma~\ref{examp2b}
(note that any $(r,s,t)$ can be replaced with $(0,0,0)$
without changing the function spaces).
If~$\beta_b$ admits product estimates,
then~$\beta_b$ is continuous, as observed in the introduction.

\emph{Case}~3: \emph{$G$ is an infinite compact group}
(and hence not discrete).
Then $\beta_b$ is always continuous,
by \cite[Corollary~2.3]{BaG}.
If $\beta_b$ admits product estimates,
then also $b$ admits these (by Lemma~\ref{preforint}),
and if $t=\infty$, then also $r=s=\infty$ (see Lemma~\ref{examp2}).
Thus (a)--(c) from Theorem~A are satisfied.
If, conversely, (a)--(c) are satisfied,
then $\beta_b$ admits product estimates, by Lemma~\ref{examp2b}.

\emph{Case}~4: \emph{$G$ is neither compact nor discrete.}
If $\beta_b$ admits product estimates, then $\beta_b$
is continuous and hence (a)--(c) hold by Theorem~A.
If, conversely, (a)--(c) are satisfied, then $\beta_b$ admits product
estimates, by Lemma~\ref{examp2b}.\,\Punkt
\section{Product estimates on spaces without norm}
If we start with a continuous bilinear map $b\colon E_1\times E_2\to F$ on a product of normed spaces,
then it satisfies product estimates (by Proposition~\ref{examp1}),
and can be fed into Theorem~C,
to obtain bilinear maps $\beta$ on function spaces that admit product estimates.
Since $E_1$ and $E_2$ are normed, also the function spaces admit a continuous norm.
However, the existence of a continuous norm on the domain $E_1\times E_2$
is not necessary for the existence of product estimates,
as the trivial example $\beta\colon E_1\times E_2\to \R$, $\beta(x,y):=0$ shows.
The situation does not change if one assumes that $\beta$ is non-degenerate
in the sense that $\beta(x,.)\not=0$ and $\beta(.,y)\not=0$ for all $0\not= x\in E_1$ and $0\not= y\in E_2$,
as illustrated by the following example.
\begin{example}\label{exenew}
Let $M$ be an uncountable set and
$E:=\R^{(M)}$ be the set of all functions $\gamma\colon M\to\R$
with finite support, equipped with the (unusual!) locally convex topology $\cO$
which is initial with respect to the restriction maps
\[
\rho_C\colon E\to \R^{(C)},\quad \gamma\mto\gamma|_C
\]
for all countable subsets $C\sub M$, where $\R^{(C)}$ is equipped
with the finest locally convex topology (turning $\R^{(C)}$ into
the locally convex direct sum $\bigoplus_{j\in C}\R$).
Hence the seminorms
\[
p_v\colon E\to[0,\infty[\,,\quad p_v(\gamma):=\max\{v(m)|\gamma(m)|\colon m\in M\}
\]
define the locally convex topology on~$E$,
for $v$ ranging through the set $\cV$ of all functions
$v\colon M\to [0,\infty[$ such that $\{m\in M\colon v(m)>0\}$ is countable.
Since none of these $p_v$ is a norm, we conclude that~$E$ does not admit a continuous
norm.
If $B\sub E$ is bounded, then $B\sub \R^F$ for some finite subset $F\sub M$,
as is easy to see.\footnote{If not, there are $\gamma_1,\gamma_2,\ldots\in B$ such that
the set $C:=\{m \in M\colon \gamma_n(m)\not=0$
for some $n\in\N\}$
is infinite. Note that~$C$ is also countable.
Now $\rho_C(B)$ is a bounded subset in $\R^{(C)}$
and hence contained in $\R^F$ for some finite subset $F\sub C$ (see \cite[II.6.3]{SaW}),
a contradiction.}
Hence $E$ is quasi-complete (and hence sequentially complete).
Consider the map
$\beta\colon E\times E\to E$,
$(\gamma,\eta)\mto\gamma \eta$
taking~$\gamma$ and~$\eta$ to their pointwise product $\gamma\eta$,
given by $(\gamma\eta)(m):=\gamma(m)\eta(m)$.
Then $\beta$
is bilinear and non-degenerate, as $\gamma\gamma \not=0$
for each $\gamma\in E\setminus\{0\}$.
The pointwise multiplication map
$\beta_C\colon \R^{(C)}\times \R^{(C)}\to \R^{(C)}$ is bilinear
and hence continuous (see Corollary~\ref{Rinfty}),
for each countable set $C\sub M$.
Since $\R^{(\N)}$
satisfies the countable upper bound condition by Proposition~\ref{excubc}\,(f),
$\beta_C$
satisfies product estimates (by Proposition~\ref{cubcuseful}).
To see that~$\beta$ admits product estimates,
let $p_{i,j}$ be continuous seminorms on~$E$ for $i,j\in \N$.
After replacing the latter by larger seminorms,
we may assume that $p_{i,j}=p_{v_{i,j}}$ for some $v_{i,j}\in \cV$.
Then $C:=\{m \in M \colon \mbox{$v_{i,j}(m)\not=0$ for some $(i,j)\in \N\times \N$}\}$
is a countable set.
Let $\cV_C$ be the set
of all $v\in \cV$
such that $v(m)=0$ for all $m\in M\setminus C$.
For $v\in \cV_C$, define $q_v\colon \R^{(C)}\to [0,\infty[$, $q_v(\gamma):=\max\{v(m)|\gamma(m)|\colon m\in C \}$.
Then $q_v\circ \rho_C=p_v$.
Since $\beta_C$ admits product estimates,
there are $v_i,w_i\in \cV_C$ such that
$q_{v_{i,j}}(\beta_C(\gamma,\eta))\leq q_{v_i}(\gamma)q_{w_j}(\eta)$
for all $i,j\in \N$ and $\gamma,\eta\in \R^{(C)}$.
Then
$p_{i,j}(\beta(\gamma,\eta))=p_{v_{i,j}}(\beta(\gamma,\eta))
=q_{v_{i,j}}(\rho_C(\beta(\gamma,\eta)))
=q_{v_{i,j}}(\beta_C(\rho_C(\gamma),\rho_C(\eta)))
\leq q_{v_i}(\rho_C(\gamma))q_{w_j}(\rho_C(\eta))
=p_{v_i}(\gamma)p_{w_j}(\eta)$
for all $\gamma,\eta\in E$,
showing that $\beta$ admits product estimates.
\end{example}
\emph{Acknowledgement.}
The research was supported by DFG, grant GL 357/5--1.
The author is grateful to the referee for
improvements of the presentation.
Thanks are also due to C. Bargetz for
discussions and to J. Bonet, E. Jord\'{a} and J. Wengenroth
for references to the literature,
which were incorporated into the final version.\vspace{1mm}
{\small Helge Gl\"{o}ckner, Universit\"{a}t Paderborn, Institut f\"{u}r Mathematik,\\
Warburger Str.\ 100, 33098 Paderborn, Germany.\\[1mm]
Email: {\tt  glockner\at{}math.upb.de}}
\end{document}